\documentclass[12pt]{amsart}
\usepackage{amsfonts,amssymb,amscd,amsmath,enumerate,verbatim}
\usepackage[latin1]{inputenc}
\usepackage{amscd}
\usepackage{tikz}
\usepackage{latexsym}
\usepackage{mathptmx}
\usepackage{multicol}
\input xy
\xyoption{all}

\def\NZQ{\mathbb}               

\def\ZZ{{\NZQ Z}}
\def\RR{{\NZQ R}}

\def\frk{\mathfrak}               

\def\Phi{{\frk N}}
\def\ab{{\bold a}}

\def\eb{{\bold e}}
\def\tb{{\bold t}}
\def\xb{{\bold x}}

\def\opn#1#2{\def#1{\operatorname{#2}}} 
\opn\chara{char} 
\opn\length{\ell} 
\opn\pd{pd} 
\opn\rk{rk}
\opn\projdim{proj\,dim} 
\opn\injdim{inj\,dim} 
\opn\rank{rank}
\opn\depth{depth} 
\opn\grade{grade} 
\opn\height{height}
\opn\embdim{emb\,dim} 
\opn\codim{codim}

\opn\Tr{Tr} 
\opn\bigrank{big\,rank}
\opn\superheight{superheight}
\opn\lcm{lcm}
\opn\trdeg{tr\,deg}
\opn\reg{reg} 
\opn\lreg{lreg} 
\opn\ini{in} 
\opn\lpd{lpd}
\opn\size{size}
\opn\mult{mult}
\opn\dist{dist}
\opn\cone{cone}
\opn\lex{lex}
\opn\rev{rev}
\opn\codeg{codeg}
\opn\div{div} \opn\Div{Div} \opn\cl{cl} \opn\Cl{Cl}
\opn\Spec{Spec} \opn\Supp{Supp} \opn\supp{supp} \opn\Sing{Sing}
\opn\Ass{Ass} \opn\Min{Min}
\opn\Ann{Ann} \opn\Rad{Rad} \opn\Soc{Soc}
\opn\Syz{Syz} \opn\Im{Im} \opn\Ker{Ker} \opn\Coker{Coker}
\opn\Am{Am} \opn\Hom{Hom} \opn\Tor{Tor} \opn\Ext{Ext}
\opn\End{End} \opn\Aut{Aut} \opn\id{id} \opn\ini{in}

\opn\nat{nat}
\opn\pff{pf}
\opn\Pf{Pf} \opn\GL{GL} \opn\SL{SL} \opn\mod{mod} \opn\ord{ord}
\opn\Gin{Gin}
\opn\Hilb{Hilb}\opn\adeg{adeg}\opn\std{std}\opn\ip{infpt}
\opn\Pol{Pol}
\opn\sat{sat}
\opn\Var{Var}
\opn\Gen{Gen}

\opn\aff{aff} \opn\con{conv} \opn\relint{relint} \opn\st{st}
\opn\lk{lk} \opn\cn{cn} \opn\core{core} \opn\vol{vol}
\opn\link{link} \opn\star{star}
\opn\gr{gr}

\def\Pc{{\mathcal P}}

\def\pot#1#2{#1[\kern-0.28ex[#2]\kern-0.28ex]}

\opn\dirlim{\underrightarrow{\lim}}
\opn\inivlim{\underleftarrow{\lim}}
\let\to=\rightarrow

\def\Implies{\ifmmode\Longrightarrow \else
        \unskip${}\Longrightarrow{}$\ignorespaces\fi}
\def\implies{\ifmmode\Rightarrow \else
        \unskip${}\Rightarrow{}$\ignorespaces\fi}
\def\iff{\ifmmode\Longleftrightarrow \else
        \unskip${}\Longleftrightarrow{}$\ignorespaces\fi}

\let\:=\colon
\newtheorem{Theorem}{Theorem}[section]
\newtheorem{Lemma}[Theorem]{Lemma}

\newtheorem{Conjecture}[Theorem]{Conjecture}

\let\epsilon\varepsilon
\let\phi=\varphi
\let\kappa=\varkappa
\def\qed{\ifhmode\textqed\fi
      \ifmmode\ifinner\quad\qedsymbol\else\dispqed\fi\fi}
\def\textqed{\unskip\nobreak\penalty50
       \hskip2em\hbox{}\nobreak\hfil\qedsymbol
       \parfillskip=0pt \finalhyphendemerits=0}
\def\dispqed{\rlap{\qquad\qedsymbol}}

\opn\dis{dis}
\opn\height{height}
\opn\dist{dist}
\def\pnt{{\raise0.5mm\hbox{\large\bf.}}}

\opn\Lex{Lex}

\begin{document}
\title{Edge rings of bipartite graphs with linear resolutions}
\author{Akiyoshi Tsuchiya}
\address[Akiyoshi Tsuchiya]
{Graduate school of Mathematical Sciences,
University of Tokyo,
Komaba, Meguro-ku, Tokyo 153-8914, Japan} 
\email{akiyoshi@ms.u-tokyo.ac.jp}
\begin{abstract}
Ohsugi and Hibi characterized the edge ring of a finite connected simple graph with a $2$-linear resolution. On the other hand, Hibi, Matsuda and the author conjectured that the edge ring of a finite connected simple graph with a $q$-linear resolution, where $q \geq 3$, is a hypersurface and proved the case $q=3$. In the present paper, we solve this conjecture for the case of finite connected simple bipartite graphs. \end{abstract}
\maketitle
\section*{Introduction}
Let $S = K[x_{1}, \ldots, x_{n}]$ denote the polynomial ring in $n$ variables over a field $K$ with each $\deg x_{i} = 1$.  Let $I \subset S$ be a homogeneous ideal of $S$ and 
\[
0 \to \bigoplus_{j \geq 1} S(-j)^{\beta_{h, j}} \to \cdots \to \bigoplus_{j \geq 1} S(-j)^{\beta_{1, j}} \to S \to S/I \to 0
\]
a (unique) graded minimal free $S$-resolution of $S/I$.  The ({\em Castelnuovo-Mumford}\,) {\em regularity} of $S/I$ is 
\[
\reg (S/I) = \max\{ j - i : \beta_{i, j} \neq 0 \}. 
\] 
We say that $S/I$ has a {\em $q$-linear resolution} if $\beta_{i, j} = 0$ for each $1 \leq i \leq h$ and for each $j \neq q + i - 1$.  If $S/I$ has a $q$-linear resolution, then $\reg(S/I) = q - 1$ and $I$ is generated by homogeneous polynomials of degree $q$.  We refer the reader to, e.g., \cite{BH} and \cite{HHO} for the detailed information about regularity and linear resolutions.

The edge ring and the edge polytope of a finite connected simple graph together with its toric ideal has been studied by many articles.  Their foundation was established in \cite{OH, 2linear}. Moreover, several papers on the minimal free resolutions of edge rings are published, see e.g. \cite{BOV,GM,HKO,HMT,Koszulbipartite,2linear}. In particular, in \cite[Theorem 4.6]{2linear} it is shown that the edge ring $K[G]$, where $K$ is a field, of a finite connected simple graph $G$ on $[N] = \{ 1, \ldots, N \}$ has a $2$-linear resolution if and only if $K[G]$ is isomorphic to the polynomial ring in $N - \delta$ variables over the Segre product $K[x_1, x_2] \sharp K[y_1, \ldots, y_\delta]$ of two polynomial rings $K[x_1, x_2]$ and $K[y_1, \ldots, y_\delta]$, where $\delta$ is the normalized volume (\cite[p.~36]{Stu}) of the edge polytope $\Pc_G$ of $G$.  On the other hand, in \cite[Theorem 0.1]{HMT} it is shown that if the edge ring $K[G]$ of a finite connected simple graph $G$ has a $3$-linear resolution, then $K[G]$ is a hypersurface. Moreover, the following conjecture is given.
\begin{Conjecture}[{\cite[Conjecture 0.2]{HMT}}]
	\label{conj}
The edge ring of a finite connected simple graph with a $q$-linear resolution, where $q \geq 3$, is a hypersurface.
\end{Conjecture}

In the present paper, we solve Conjecture \ref{conj} for the case of bipartite graphs. In fact,
\begin{Theorem}
	\label{thm:main}
	The edge ring of a finite connected simple bipartite graph with a $q$-linear resolution, where $q \geq 3$, is a hypersurface.
\end{Theorem}
The edge rings of bipartite graphs and the associated lattice polytopes, which are called the edge polytopes (or root polytopes) play particularly important roles in commutative algebra and combinatorics, e.g., see \cite{KalPos,Koszulbipartite}.

In the present paper, after preparing necessary materials on edge polytopes and edge rings (Section \ref{sec:edge}), Theorem \ref{thm:main} will be proved in Section \ref{sec:proof}.

\subsection*{Acknowledgment}
The author was partially supported by JSPS KAKENHI 19K14505 and 19J00312.

\section{Edge polytopes and Edge rings}
\label{sec:edge}
A {\em lattice polytope} is a convex polytope all of whose coordinates have integer coordinates.  Let $\Pc \subset \RR^N$ be a lattice polytope of dimension $d$ and $\Pc \cap \ZZ^N=\{\ab_1,\ldots,\ab_n \}$.  Let $K$ be a field and $K[\tb^{\pm1}, s]=K[t_1^{\pm1},\ldots,t_N^{\pm 1},s]$ the Laurent polynomial ring in $N + 1$ variables over $K$.  Given a lattice point  $\ab=(a_1,\ldots,a_N) \in \ZZ^N$, we write $\tb^{\ab}$ for the Laurent monomial $t_1^{a_1}\cdots t_N^{a_N} \in K[\tb^{\pm 1}, s]$.  The \textit{toric ring} $K[\Pc]$ of $\Pc$ is the subalgebra of $K[\tb^{\pm 1},s]$ which is generated by those monomials $\tb^{\ab_1}s,\ldots,\tb^{\ab_n}s$ over $K$.  Let $K[\xb]=K[x_1,\ldots,x_n]$ denote the polynomial ring in $n$ variables over $K$ with each $\deg x_i=1$, and define the surjective ring homomorphism $\pi:K[\xb] \to K[\Pc]$ by setting $\pi(x_i)=\tb^{\ab_i}s$ for $1 \leq i \leq n$.  The kernel $I_{\Pc}$ of $\pi$ is called the \textit{toric ideal} of $\Pc$.

Let $G$ be a finite simple graph on the vertex set $V(G)=[N]$ with the edge set $E(G)=\{e_1,\ldots,e_n\}$.  (A finite graph $G$ is called simple if $G$ possesses no loop and no multiple edge.) Let $\eb_1,\ldots,\eb_N$ denote the canonical unit coordinate vectors of $\RR^N$.  Given an edge $e=\{i,j\}$ of $G$, we set $\rho(e)=\eb_i+\eb_j \in \RR^{N}$.  The \textit{edge polytope}  $\Pc_G$ of $G$ is the lattice polytope which is the convex hull of $\{\rho(e_1),\ldots,\rho(e_n)\}$ in $\RR^N$.  
The edge polytope $\Pc_G$ of a finite simple bipartite graph $G$ is also called the \textit{root polytope} of $G$. 
One has $\dim \Pc_G = N - c_0(G)-1$, where $c_0(G)$ is the number of connected bipartite components of $G$ (\cite[p. 57]{VV}). In particular, if $G$ is connected and bipartite, then $\dim \Pc_G = N - 2$.
%
%
%
The \textit{edge ring} $K[G]$ of $G$ is the toric ring of $\Pc_G$, that is, $K[G]=K[\Pc_G]$ and the toric ideal $I_G$ of $K[G]$ is the toric ideal of $\Pc_{G}$, that is, $I_{G} = I_{\Pc_{G}}$. 
When the edge ring $K[G]$ of a finite simple graph $G$ is studied, we follow the convention of assuming that $G$ is connected.  Let $G$ be a finite disconnected simple graph with the connected components $G_1, G_2, \ldots, G_s$ and suppose that each $G_i$ has at least one edge.  Then the edge ring of $G$ is 
$
K[G] = K[G_{1}] \otimes_{K} \cdots \otimes_{K} K[{G_{s}}]
$ and its toric ideal is
\[
(I_{G_1}, I_{G_2}, \ldots, I_{G_s})
\subset K[x_1^{(1)}, \ldots, x_{n_1}^{(1)}, x_1^{(2)}, \ldots, x_{n_2}^{(2)}, \ldots, x_1^{(s)}, \ldots, x_{n_s}^{(s)}].
\]
Let, say, $I_{G_1} \neq (0)$ and $I_{G_2} \neq (0)$.  Then $K[G]$ cannot have a linear resolution.  Hence $K[G]$ has a $q$-linear resolution if and only if there is $1 \leq i \leq s$ for which $K[G_i]$ has a $q$-linear resolution and each $K[G_j]$ with $i \neq j$ is the polynomial ring.

Recall from \cite{2linear} what a system of generators of the toric ideal $I_G$ of a finite connected simple bipartite graph is.  
Let $C$ be an even cycle of $G$ with the edge set 
\[
E(C)=\{e_{i_1},\ldots,e_{i_{2n}}\},
\]
where for $1 \leq j \leq 2n-1$, $e_{i_j}=\{v_{j},v_{j+1}\}$ and $e_{i_{2n}}=\{v_{2n},v_{1}\}$.
We write $f_C$ for the binomial
\[
f_C=\prod_{j=1}^{n}x_{i_{2j-1}}-\prod_{j=1}^{n}x_{i_{2j}}
\]
belonging to $I_G$, where $\pi(x_i)=\tb^{\rho(e_i)}s$.
\begin{Lemma}[{\cite[Corollary 5.12]{HHO}}]
	The toric ideal $I_G$ of a finite connected simple bipartite graph $G$ is generated by those binomials $f_{C}$, where $C$ is an even cycle of $G$.
\end{Lemma}

As a result, it follows that

\begin{Lemma}\label{lem:gen}

Let $G$ be a finite connected simple bipartite graph. Assume that $I_G$ is generated by homogeneous polynomials of degree $n$. Then $G$ has no even cycles of length $< 2n$. In particular, $I_G$ is generated by $\{f_{C} : C \mbox{ is an even cycle of length } 2n\}$.
\end{Lemma}

Next, we recall Ehrhart theory of lattice polytopes. Let $\Pc \subset \RR^N$ be a lattice polytope of dimension $d$.  The {\em $h^*$-polynomial} (or {\em $\delta$-polynomial}) of $\Pc$ is the polynomial
\[
h^*(\Pc,\lambda)=(1-\lambda)^{d+1} \left[ 1+ \sum_{t=1}^{\infty} |t\Pc \cap \ZZ^N | \lambda^t \right]
\]
in $\lambda$, where $t\Pc=\{t \ab : \ab \in \Pc \}$.  Each coefficient of $h^*(\Pc, \lambda)$ is a nonnegative integer and the degree of $h^*(\Pc, \lambda)$ is at most $d$. Let $\deg(\Pc)$ denote the degree of $h^*(\Pc,\lambda)$ and set $\codeg(\Pc)=d+1-\deg(\Pc)$.
It then follows that
\[
\codeg(\Pc)= \min \{r \in \ZZ_{\geq 1} : \text{int}(r\Pc)\cap \ZZ^{N} \neq \emptyset \},
\]
where $\text{int}(\Pc)$ is the relative interior of $\Pc$ in $\RR^N$ and where $\ZZ_{\geq 1}$ stands for the set of positive integers.  We refer the reader to \cite[Part II]{HibiRedBook} for the detailed information about $\delta$-polynomials and their related topics.

From 
\cite{S}  we obtain the following:
\begin{Lemma}[{\cite[Corollary 3.2]{HMT}}]
	\label{lem:deg}
	Let $G$ be a finite connected simple graph and let $G' \subset G$ be a  subgraph of $G$.  Then $\deg(\Pc_{G'}) \leq \deg (\Pc_{G})$. 
\end{Lemma}

On the other hand, from {\cite[p. 5952]{HKN}}
one has the following:
\begin{Lemma}[{\cite[Corollary 3.4]{HMT}}]
	\label{lem:reg}
	Let $G$ be a finite connected simple graph on $[N]$.
	Then 
	\[
	\reg (K[\Pc_G]) \ge \deg (\Pc_G).
	\] 
\end{Lemma}

\section{Proof of main theorem}
\label{sec:proof}
Let $G$ be a finite connected simple bipartite graph on $[N]$ such that $K[G]$ has a $q$-linear resolution with $q \geq 3$. 
Then it follows that ${\rm reg}(K[G])=q-1$ and $I_G$ is generated by homogeneous polynomials of degree $q$. Hence in order to prove Theorem \ref{thm:main}, from Lemma \ref{lem:gen}, we should show that	 if $G$ has no even cycles of length $< 2q$ and has at least two even cycles of length $2q$,
	then we obtain ${\rm reg}(K[G]) \geq q$.

\begin{Lemma}
\label{lem:disjoint}
	If $G$ has disjoint two even cycles of length $2q$, then we obtain $\deg \Pc_G \geq q$.
\end{Lemma}

\begin{proof}
	Let $G'$ be a finite simple graph on $[4q]$ with the edge set
	$E(G')=\{ e_1,\ldots,e_{4q}\}$,
	where $e_i=\{i,i+1\}$ for $1 \leq i \leq 4q-1$ with $i \neq 2q$ and $e_{2q}=\{2q,1\}$ and $e_{4q}=\{4q,2q+1\}$.
	Namely, $G'$ is the disjoint union of two even cycles of length $2q$.
	Then one has $\dim \Pc_{G'}=4q-3$. 
	Since 
	\[
	\dfrac{1}{2}\sum_{i=1}^{4q} \rho (e_i)=\eb_1+\cdots+\eb_{4q} \in {\rm int}(2q \Pc_{G'}) \cap \ZZ^{4q},
	\]
	we obtain  ${\rm codeg}(\Pc_{G'}) \leq 2q$.
	Hence it follows that
	\[
	{\rm deg}(\Pc_{G'})=\dim \Pc_{G'} + 1 - {\rm codeg}(\Pc_{G'}) \geq 4q-3+1-2q=2q-2 \geq q.
	\]
	Therefore, from Lemma \ref{lem:deg}, one has $\deg (\Pc_G) \geq \deg(\Pc_{G'}) \geq q$, as desired.
\end{proof}

\begin{Lemma}
\label{lem:twocycles}
If $G$ has two even cycles of length $2q$ which have precisely one common vertex, then we obtain $\deg \Pc_G \geq q$.
\end{Lemma}
\begin{proof}
 Let $G'$ be a connected finite simple graph on $[4q-1]$ with the edge set $E(G')=\{e_1,\ldots, e_{4q} \}$, where 
	\[
	e_i=\begin{cases}
		\{i,i+1\} & 1 \leq i \leq 2q-1,\\
		\{2q,1 \} & i=2q,\\
		\{1,2q+1\} &i=2q+1,\\
		\{i-1,i\} & 2q+2 \leq i \leq 4q-1,\\
		\{4q-1,1\} & i=4q.
	\end{cases}
	\]

Since $G'$ is connected and bipartite, It follows that $\dim \Pc_{G_k}=4q-3$.
Since 
\begin{displaymath}
	\begin{aligned}
		&	\sum_{i=1}^{2q} \left(\dfrac{1}{3} \rho(e_{2i-1}) + \dfrac{2}{3}\rho(e_{2i})\right)\\
		=&2\eb_1+\eb_2+\cdots+\eb_{4q-1} \in {\rm int}(2q\Pc_{G'}) \cap \ZZ^{4q-1},
	\end{aligned}
\end{displaymath}
one has
	\[
	{\rm deg}(\Pc_{G'}) \geq 4q-3+1 -2q=2q-2 \geq q.
	\]
Hence from Lemma \ref{lem:deg} we obtain $\deg \Pc_G \geq \deg \Pc_{G'} \geq q$, as desired.
	\end{proof}
	
	\begin{Lemma}
\label{lem:even}
Given positive integers $k$ and $m$ with $k+m \geq q$ and $k \leq m$,  let $G^{(e)}_{k,m}$ be a connected finite simple bipartite graph on $[2q+2m-1]$ with the edge set $E(G^{(e)}_{k,m})=\{e_1,\ldots,e_{2q+2m} \}$, where 	\[
	e_i=\begin{cases}
		\{i,i+1\} & 1 \leq i \leq 2q-1,\\
		\{2q,1 \} & i=2q,\\
		\{1,2q+1\} &i=2q+1,\\
		\{i-1,i\} & 2q+2 \leq i \leq 2q+2m-1,\\
		\{2q-2m-1,2k+1\} & i=2q+2m,
	\end{cases}
	\]
see {\rm FIGURE 1}.
Then we obtain $\deg \Pc_{G^{(e)}_{k,m}} \geq q$.
	\end{Lemma}
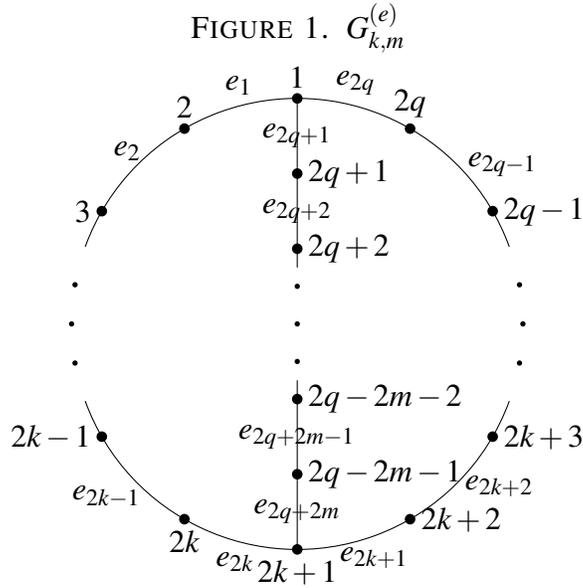
		\begin{figure}[h]
	\label{fig:even}
	\caption{$G^{(e)}_{k,m}$}
	\begin{tikzpicture}
	\coordinate (v1) at (90:3) node at (v1) [above] {$1$};
	\coordinate (v2) at (120:3) node at (v2) [above] {$2$};
	\coordinate (v3) at (150:3) node at (v3) [left] {$3$};
	\coordinate (v4) at (210:3) node at (v4) [left] {$2k-1$};
	\coordinate (v5) at (240:3) node at (v5) [below] {$2k$};
	\coordinate (v6) at (270:3) node at (v6) [below] {$2k+1$};
	\coordinate (v7) at (300:3) node at (v7) [right] {$2k+2$};
	\coordinate (v8) at (330:3) node at (v8) [right] {$2k+3$};
	\coordinate (v9) at (30:3) node at (v9) [right] {$2q-1$};
	\coordinate (v10) at (60:3) node at (v10) [above] {$2q$};
	\coordinate (v11) at (90:2) node at (v11) [right] {$2q+1$};
	\coordinate (v12) at (90:1) node at  (v12) [right] {$2q+2$};
	\coordinate (v13) at (270:1) node at (v13) [right] {$2q-2m-2$};
	\coordinate (v14) at (270:2) node at (v14) [right] {$2q-2m-1$};
	\coordinate (v15) at (90:0.75);
	\coordinate (v16) at (270:0.75);
	\coordinate (e1) at (105:3) node at (e1) [above] {$e_1$};
	\coordinate (e2) at (130:3) node at (e2) [left] {$e_2$};
	\coordinate (e3) at (230:3) node at (e3) [left] {$e_{2k-1}$};
	\coordinate (e4) at (254:3) node at (e4) [below] {$e_{2k}$};
	\coordinate (e5) at (290:3) node at (e5) [below] {$e_{2k+1}$};	
	\coordinate (e6) at (314:3) node at (e6) [right] {$e_{2k+2}$};	
	\coordinate (e7) at (45:3) node at (e7) [right] {$e_{2q-1}$};
	\coordinate (e8) at (75:3) node at (e8) [above] {$e_{2q}$};
	\coordinate (e9) at (90:2.5) node at (e9) {$e_{2q+1}$};
	\coordinate (e10) at (90:1.5) node at (e10)  {$e_{2q+2}$};
	\coordinate (e11) at (270:1.5) node at (e11)  {$e_{2q+2m-1}$};	
	\coordinate (e12) at (270:2.5) node at (e12)  {$e_{2q+2m}$};	
	\draw 
	(20:3) arc (20:160:3cm)
	(200:3) arc (200:340:3cm)
	(v1)--(v15)
	(v6)--(v16)
	;
\fill 
(v1) circle (2pt)
(v2) circle (2pt)
(v3) circle (2pt)
(v4) circle (2pt)
(v5) circle (2pt)
(v6) circle (2pt)
(v7) circle (2pt)
(v8) circle (2pt)
(v9) circle (2pt)
(v10) circle (2pt)
(v11) circle (2pt)
(v12) circle (2pt)
(v13) circle (2pt)
(v14) circle (2pt)
(180:3) circle (1pt)
(170:3) circle (1pt)
(190:3) circle (1pt)
(350:3) circle (1pt)
(0:3) circle (1pt)
(10:3) circle (1pt)
(0:0) circle (1pt)
(90:0.5) circle (1pt)
(270:0.5) circle (1pt)
;
    \end{tikzpicture}	
	\end{figure}
	
	\begin{proof}
	Since $\Pc_{G^{(e)}_{k,m}}$ is bipartite and connected, one has $\dim \Pc_{G^{(e)}_{k,m}}=2q+2m-3$.
		Moreover, since
	\begin{displaymath}
	\begin{aligned}
		&	\sum_{i=1}^{k} \left(\dfrac{1}{3} \rho(e_{2i-1}) + \dfrac{2}{3}\rho(e_{2i}) \right)\\+&\sum_{i=k+1}^{q} \left(\dfrac{2}{3} \rho(e_{2i-1}) + \dfrac{1}{3}\rho(e_{2i})\right)
		\\+&\sum_{i=q+1}^{q+m} \left(\dfrac{1}{3} \rho(e_{2i-1}) + \dfrac{2}{3}\rho(e_{2i})\right)	\\=&\eb_1+\eb_2+\cdots+\eb_{2q+2m-1}+\eb_{2k+1} \in {\rm int}((q+m)\Pc_{G^{(e)}_{k,m}}) \cap \ZZ^{2q+2m-1}
	\end{aligned}
\end{displaymath}
and since $m \geq q/2$, 
we obtain $\deg \Pc_{G^{(e)}_{k,m}} \geq 2q+2m-3+1-q-m=q+m-2 \geq  3q/2 -2$.
If $q=3$, then $\deg \Pc_{G^{(e)}_{k,m}} \geq 3 > 5/2$. Moreover, if $q \geq 4$, then $\deg \Pc_{G^{(e)}_{k,m}} 3q/2 -2 \geq q$, as desired.
	\end{proof}
	
	\begin{Lemma}
		\label{lem:odd}
		Given positive integers $k$ and $m$ with $k+m-1 \geq q$ and $k \leq m$,  let $G^{(o)}_{k,m}$ be a connected finite simple bipartite graph on $[2q+2m-2]$ with the edge set $E(G^{(o)}_{k,m})=\{e_1,\ldots,e_{2q+2m-1} \}$, where 	\[
	e_i=\begin{cases}
		\{i,i+1\} & 1 \leq i \leq 2q-1,\\
		\{2q,1 \} & i=2q,\\
		\{1,2q+1\} &i=2q+1,\\
		\{i-1,i\} & 2q+2 \leq i \leq 2q+2m-2,\\
		\{2q-2m-2,2k\} & i=2q-2m-1,
	\end{cases}
	\]
see {\rm FIGURE 2}.
Then we obtain $\deg \Pc_{G^{(o)}_{k,m}} \geq q$.
	\end{Lemma}
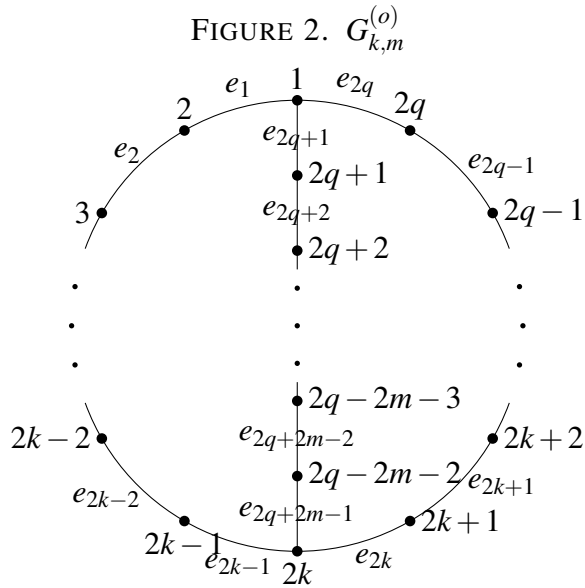
	\begin{figure}[h]
	\label{fig:odd}
	\caption{$G^{(o)}_{k,m}$}
	\begin{tikzpicture}
	\coordinate (v1) at (90:3) node at (v1) [above] {$1$};
	\coordinate (v2) at (120:3) node at (v2) [above] {$2$};
	\coordinate (v3) at (150:3) node at (v3) [left] {$3$};
	\coordinate (v4) at (210:3) node at (v4) [left] {$2k-2$};
	\coordinate (v5) at (240:3) node at (v5) [below] {$2k-1$};
	\coordinate (v6) at (270:3) node at (v6) [below] {$2k$};
	\coordinate (v7) at (300:3) node at (v7) [right] {$2k+1$};
	\coordinate (v8) at (330:3) node at (v8) [right] {$2k+2$};
	\coordinate (v9) at (30:3) node at (v9) [right] {$2q-1$};
	\coordinate (v10) at (60:3) node at (v10) [above] {$2q$};
	\coordinate (v11) at (90:2) node at (v11) [right] {$2q+1$};
	\coordinate (v12) at (90:1) node at  (v12) [right] {$2q+2$};
	\coordinate (v13) at (270:1) node at (v13) [right] {$2q-2m-3$};
	\coordinate (v14) at (270:2) node at (v14) [right] {$2q-2m-2$};
	\coordinate (v15) at (90:0.75);
	\coordinate (v16) at (270:0.75);
	\coordinate (e1) at (105:3) node at (e1) [above] {$e_1$};
	\coordinate (e2) at (130:3) node at (e2) [left] {$e_2$};
	\coordinate (e3) at (230:3) node at (e3) [left] {$e_{2k-2}$};
	\coordinate (e4) at (255:3) node at (e4) [below] {$e_{2k-1}$};
	\coordinate (e5) at (290:3) node at (e5) [below] {$e_{2k}$};	
	\coordinate (e6) at (315:3) node at (e6) [right] {$e_{2k+1}$};	
	\coordinate (e7) at (45:3) node at (e7) [right] {$e_{2q-1}$};
	\coordinate (e8) at (75:3) node at (e8) [above] {$e_{2q}$};
	\coordinate (e9) at (90:2.5) node at (e9)  {$e_{2q+1}$};
	\coordinate (e10) at (90:1.5) node at (e10)  {$e_{2q+2}$};
	\coordinate (e11) at (270:1.5) node at (e11)  {$e_{2q+2m-2}$};	
	\coordinate (e12) at (270:2.5) node at (e12)  {$e_{2q+2m-1}$};	
	\draw 
	(20:3) arc (20:160:3cm)
	(200:3) arc (200:340:3cm)
	(v1)--(v15)
	(v6)--(v16)
	;
\fill 
(v1) circle (2pt)
(v2) circle (2pt)
(v3) circle (2pt)
(v4) circle (2pt)
(v5) circle (2pt)
(v6) circle (2pt)
(v7) circle (2pt)
(v8) circle (2pt)
(v9) circle (2pt)
(v10) circle (2pt)
(v11) circle (2pt)
(v12) circle (2pt)
(v13) circle (2pt)
(v14) circle (2pt)
(180:3) circle (1pt)
(170:3) circle (1pt)
(190:3) circle (1pt)
(350:3) circle (1pt)
(0:3) circle (1pt)
(10:3) circle (1pt)
(0:0) circle (1pt)
(90:0.5) circle (1pt)
(270:0.5) circle (1pt)
;
    \end{tikzpicture}	
	\end{figure}
	
	\begin{proof}
			Since $\Pc_{G^{(o)}_{k,m}}$ is bipartite and connected, one has $\dim \Pc_{G^{(o)}_{k,m}}=2q+2m-4$.
		Moreover, since
	\begin{displaymath}
	\begin{aligned}
		&	\sum_{i=1}^{k-1} \left(\dfrac{1}{3} \rho(e_{2i-1}) + \dfrac{2}{3}\rho(e_{2i}) \right)+\dfrac{1}{3}\rho(e_{2k-1})
		\\+&\sum_{i=k}^{q-1} \left(\dfrac{1}{3} \rho(e_{2i}) + \dfrac{2}{3}\rho(e_{2i+1})\right)+\dfrac{1}{3}\rho(e_{2q})\\
		+&\sum_{i=q}^{q+m-2} \left(\dfrac{1}{3} \rho(e_{2i+1}) + \dfrac{2}{3}\rho(e_{2i+2})\right)+\dfrac{1}{3}\rho(e_{2q+2m-1})	\\=&\eb_1+\eb_2+\cdots+\eb_{2q+2m-2}\in {\rm int}((q+m-1)\Pc_{G^{(o)}_{k,m}}) \cap \ZZ^{2q+2m-2}
	\end{aligned}
\end{displaymath}
and since $m \geq (q+1)/2$, 
we obtain $\deg \Pc_{G^{(o)}_{k,m}} \geq 2q+2m-4+1-q-m+1=q+m-2 \geq  3q/2 -3/2 \geq q$, as desired.
	\end{proof}
	
	Now, we prove Theorem \ref{thm:main}.
	\begin{proof}[Proof of Theorem \ref{thm:main}]
	Let $G$ be a connected finite simple bipartite graph such that $K[G]$ has a $q$-linear resolution, where $q \geq 3$. 
	It then follows that $G$ has no even cycles of length $< 2q$ and has at least one even cycle of length $2q$, and $\deg \Pc_G \leq q-1$ from Lemma \ref{lem:reg}.
	Assume that $G$ has two even cycles of length $2q$.
	From Lemma \ref{lem:disjoint} each pair of even cycles of length $2q$ in $G$ has a common vertex. If two even cycles of length $2q$ in $G$ has precisely one common vertex. It then from Lemma \ref{lem:twocycles} that $\deg \Pc_G \geq q$, a contradiction.
	Hence each pair of even cycles of length $2q$ in $G$ has at least two common vertices.
	Let $C_1$ and $C_2 (C_1 \neq C_2)$ be even cycles of length $2q$ in $G$ with common vertices $v_1$ and $v_2 (v_1 \neq v_2)$. 
Then each of $C_1$ and $C_2$ has two paths between $v_1$ and $v_2$.
Let $P_{i1}$ and $P_{i2}$ be two paths between $v_1$ and $v_2$ in $C_i$ for each $i=1,2$. Since $C_1 \neq C_2$, one has $\{P_{11},P_{12,}\} \neq \{P_{21}, P_{22} \}$.	Hence there are two vertices $w_1,w_2 \in C_1$ and a path $P$ of $G$ connecting $w_1$ and $w_2$ such that for each $w \in P$ with $w \neq w_1,w_2$, it holds $w \notin C_1$. 
Let $G'$ be the connected bipartite graph on $C_1 \cup P$ with the edge set $E(G')=E(C_1) \cup E(P)$.
Since $G$ has no odd cycles and no even cycles of length $<2n$, so does $G'$.
Hence $G'$ is $G^{(e)}_{k,m}$ or $G^{(o)}_{k,m}$ which appear in Lemmas \ref{lem:even} and \ref{lem:odd}, and hence one has $\deg \Pc_G \geq \deg \Pc_{G'} \geq q$, a contradiction.
Therefore, $G$ has precisely one even cycle of length $2q$.
Thus, $K[G]$ is a hypersurface, as desired. 
\end{proof}

\end{document}